\documentclass[a4paper,twoside]{article}      
\usepackage{amsmath,amssymb,amsfonts,amsthm,amscd,mathtools}
\usepackage{graphics}                 
\usepackage{color}                    
\usepackage{hyperref}                
\usepackage{ mathrsfs }
\usepackage{indentfirst}
\usepackage{bbold}
\usepackage{bm}
\usepackage{enumerate}
\usepackage{url}         
\usepackage{colonequals} 
\usepackage{authblk} 
\usepackage{a4wide}
\usepackage{graphicx}
\hypersetup{
    bookmarks=true,         
    unicode=false,          
    pdftoolbar=true,        
    pdfmenubar=true,        
    pdffitwindow=false,     
    pdfstartview={FitH},    
    pdftitle={My title},    
    pdfauthor={Author},     
    pdfsubject={Subject},   
    pdfcreator={Creator},   
    pdfproducer={Producer}, 
    pdfkeywords={keywords}, 
    pdfnewwindow=true,      
    colorlinks=true,       
    linkcolor=blue,          
    citecolor=blue,        
    filecolor=magenta,      
    urlcolor=cyan           
}

\newtheorem{theorem}{Theorem}[section]
\newtheorem{proposition}[theorem]{Proposition}
\newtheorem{corollary}[theorem]{Corollary}
\newtheorem{lemma}[theorem]{Lemma}
\theoremstyle{definition}
\newtheorem{remark}[theorem]{Remark}

\def\!{\mathop{\mathrm{!}}}

\def\x{\textbf{x}}
\def\j{\textbf{j}}

\def\f{\bar{f}}

\def\c{\mathbf{c}}

\newlength{\boxwidth}
\author[1]{Manh Hong Duong}
\author[2]{The Anh Han}
\affil[1]{School of Mathematics, University of Birmingham, Birmingham B15 2TT, UK. Email: h.duong@bham.ac.uk}
\affil[2]{School of Computing and Digital Technologies, Teesside University, TS1 3BX, UK. Email: T.Han@tees.ac.uk}
\title{On equilibrium properties of  the replicator-mutator equation in deterministic and random games}
\begin{document}
\maketitle
\begin{abstract}
In this paper, we study the number of equilibria of the replicator-mutator dynamics for both deterministic and random multi-player two-strategy evolutionary games. For deterministic games, using Decartes' rule of signs, we provide a formula to compute the number of equilibria in multi-player games via the number of change of signs in the coefficients of a polynomial. For two-player social dilemmas (namely, the Prisoner's Dilemma, Snowdrift, Stag Hunt, and Harmony),  we characterize (stable) equilibrium points and analytically calculate the probability of having a certain number of equilibria when the payoff entries are uniformly distributed. For multi-player random games whose payoffs are independently distributed according to a normal distribution, by employing techniques from random polynomial theory, we compute the expected or average  number of internal equilibria. In addition, we  perform extensive simulations by  sampling and averaging over  a large number of possible payoff matrices to compare with and illustrate analytical results. Numerical simulations also suggest several interesting behaviour of the average number of equilibria when the number of players is sufficiently large or when the mutation is sufficiently small. In general, we observe that introducing mutation results in  a larger average number of internal equilibria than when mutation is absent, implying that mutation leads to larger behavioural diversity in dynamical systems. Interestingly, this number is largest when mutation is  rare rather than when it is frequent.    
\end{abstract}
\section{Introduction}
The replicator-mutator dynamics has become a powerful mathematical framework for the modelling and analysis of complex biological, economical and social systems.  It has been employed in the study of, among other applications, population genetics \cite{Hadeler1981}, autocatalytic reaction networks
\cite{StadlerSchuster1992}, language evolution
\cite{Nowaketal2001}, the evolution of cooperation \cite{Imhof-etal2005} and dynamics of behavior in social networks \cite{Olfati2007}. Suppose that in an infinite population there are $n$ types/strategies $S_1,\cdots, S_n$ whose frequencies are, respectively, $x_1,\cdots, x_n$. These types undergo selection; that is, the reproduction rate of each type, $S_i$, is determined by its fitness or average payoff, $f_i$, which is obtained from interacting with other individuals in the population. The interaction of the individuals in the population is carried out within randomly selected groups of $d$ participants (for some integer $d$). That is, they play and obtain their payoffs from a $d$-player game, defined by a payoff matrix. We consider here symmetric games where the payoffs do not depend on the ordering of the players in a group. Mutation is included by adding the possibility that individuals spontaneously change from one strategy to another, which is modeled via a mutation matrix, $Q=(q_{ji}), j,i\in\{1,\cdots,n\}$. The entry $q_{ji}$ denotes the  probability that a player of type $S_j$ changes its type or strategy to $S_i$. The mutation matrix $Q$ is a row-stochastic matrix, i.e.,
\[
\sum_{j=1}^n q_{ji}=1, \quad 1\leq i\leq n.
\]
The replicator-mutator is then given by, see e.g. \cite{Komarova2001JTB,Komarova2004, Komarova2010, Pais2012} 
\begin{equation}
\label{eq: RME}
\dot{x}_i=\sum_{j=1}^n x_j f_j(\x)q_{ji}- x_i \f(\x)=:g_i(x),\qquad  i=1,\ldots, n,
\end{equation}
where $\x = (x_1, x_2, \dots, x_n)$ and $\f(\x)=\sum_{i=1}^n x_i f_i(\x)$ denotes the average fitness of the whole population.  The replicator dynamics is a special instance of \eqref{eq: RME} when the mutation matrix is the identity matrix. 

In this paper we are interested in properties of the equilibrium points of the replicator-mutator dynamics \eqref{eq: RME}. Note that we are concerned with dynamic equilibria almost exclusively. There might be a dynamic equilibrium which is not a Nash equilibrium of the game. These dynamic equilibrium points are solutions of the following system of polynomial equations
\begin{equation}
\label{eq: pols sys}
\begin{cases}
g_i(x)=0, \quad i=1,\ldots, n-1,\\
\sum_{i=1}^n x_i=1.
\end{cases}
\end{equation}
The second condition in \eqref{eq: pols sys}, that is the preservation of the sum of the frequencies, is due to the term $x_i\f(\x)$ in \eqref{eq: RME}. The first condition imposes relations on the fitnesses. We consider both deterministic and random games where the entries of the payoff matrix are respectively deterministic and random variables. Typical examples of deterministic games include pairwise social dilemma  and public goods games that have been studied intensively in the literature, see e.g. \cite{hauert:2002te,santos:2006pn,HanJaamas2016,wang2015universal,perc2017statistical}. On the other hand, random evolutionary games are suitable for modelling social and biological systems in which very limited information is available, or where the environment changes so rapidly and frequently that one cannot describe the payoffs of their inhabitants' interactions \cite{fudenberg:1992bv, gross2009generalized, Galla2013}. 
Simulations and analysis of random games are also helpful for the prediction of the bifurcation of the replicator-mutator dynamics \cite{Komarova2001JTB, Komarova2004, Pais2012}. Here we are mainly interested in the number of equilibria in deterministic games and the expected number of equilibria in random games, which allow predicting the levels of social and biological diversity as well as the overall complexity  in a dynamical system. As in \cite{Komarova2001JTB,Komarova2004,Pais2012}, we consider an independent mutation model that corresponds to a uniform random probability of mutating to alternative strategies as follows:
\begin{equation}
\label{eq: conditionQ}
q_{ij}=\frac{q}{n-1},~~i\neq j,~~q_{ii}=1-q,~~1\leq i,j\leq n.
\end{equation}
In particular, for two-strategy games (i.e., when $n=2$), the above relations read
$$
q_{12}=q_{21}=q,~~ q_{11}=q_{22}=1-q.
$$
The parameter $q$ represents the strength of mutation and ranges from $0$ to $1-\frac{1}{n}$. The two boundaries have interesting interpretation in the context of dynamics of learning \cite{Komarova2004}:  for $q=0$ (which corresponds to the replicator dynamics), learning is perfect and learners always end up speaking the grammar of their teachers. In this case, vertices of the unit hypercube in $\mathbb{R}^n$ are always equilibria. On the other hand, for $q=\frac{n-1}{n}$, the chance for the learner to pick any grammar is the same for all grammars and is independent of the teacher's grammar. In this case, there always exists a uniform equilibrium $\x=(1/n,\ldots, 1/n)$ (cf. Remark \ref{rem: q=1/n}). Equilibrium properties of the replicator dynamics, particularly the probability of observing the maximal number of equilibrium points, the attainability and stability of the patterns of evolutionarily stable strategies have been studied intensively in the literature \cite{broom:1997aa, Broom2000,gokhale:2010pn,HTG12,gokhale2014evolutionary}. More recently, we have provided explicit formulas for the computation of the expected number and the distribution of internal equilibria for the replicator dynamics with multi-player games by employing techniques from both classical and random polynomial theory \cite{DH15, DuongHanJMB2016, DuongTranHanDGA, DuongTranHanJMB}. For the replicator dynamics, that is when there is no mutation, the first condition in \eqref{eq: pols sys} means that all the strategies have the same fitness which is also the average fitness of the whole population. This benign property is no longer valid in the presence of mutation making the mathematical analysis harder. In a general $d$-player $n$-strategy game, each $g_i$ is a multivariate polynomial of degree $d+1$, thus \eqref{eq: pols sys} is a system of multivariate polynomial equations. In particular, for a two-player two-strategy game, which is the simplest case, \eqref{eq: pols sys} reduces to a cubic equation whose coefficients depend on the payoff entries and the mutation strength. For larger $d$ and $n$, solving \eqref{eq: pols sys} analytically is generally impossible  according to Abel's impossibility theorem. Nevertheless, there has been a considerable effort to study equilibrium properties of the replicator-mutator dynamics in deterministic two-player games, see for instance \cite{Komarova2001JTB,Komarova2004, Komarova2010, Pais2012}. In particular, with the mutation strength $q$ as the bifurcation parameter, bifurcations and limit cycles have been shown for various classes of fitness matrices \cite{Komarova2010,Pais2012}. However, equilibrium properties for multi-player games and for random games are much less understood although in the previously mentioned papers, random games were employed to detect and predict certain behaviour of \eqref{eq: RME}.

In this paper, we explore further connections between classical/random polynomial theory and evolutionary game theory developed in \cite{DH15, DuongHanJMB2016, DuongTranHanDGA, DuongTranHanJMB} to study equilibrium properties of the replicator-mutator dynamics. For deterministic games, by using Decartes' rule of signs and its recent developments, we are able to fully characterize the equilibrium properties for social dilemmas. In addition, we provide a method to compute the number of equilibria in multi-player games via the sign changes of the coefficients of a polynomial. For two-player social dilemma games, we calculate the probability of having a certain number of equilibria when the payoff entries are uniformly distributed. For multi-player two-strategy random games whose payoffs are independently distributed according to a normal distribution, we obtain explicit formulas to compute the expected number of equilibria by relating it to the expected number of positive roots of a random polynomial. Interestingly,  due to mutation, the coefficients of the random polynomial become correlated as opposed to the replicator dynamics where they are independent. The case $q=0.5$ turns out to be special and needs different treatment. We also perform extensive simulations by  sampling and averaging over  a large number of possible payoff matrices, to compare with and illustrate analytical results. Moreover, numerical simulations also show interesting behaviour of the expected number of equilibria when the number of players tends to infinity or when the mutation goes to zero. It would be challenging to analyze these asymptotic behaviours rigorously and we leave it for future work.

The rest of the paper is organized as follows. In Section \ref{sec: deterministic} we study deterministic games. In Section \ref{sec: random} we consider random games. Finally, we provide further discussions and outlook in Section \ref{sec: summary}.

\section{Properties of equilibrium points: deterministic games}
\label{sec: deterministic}
In this section, we study properties of equilibrium points of deterministic games. We start with some preliminary results on the roots of a general polynomial that will be used in the subsequent sections. We then focus on two-player games, particularly the social dilemmas. Finally, by employing Decartes' rule of signs and its recent improvement \cite{Avendano2010} we derive a formula to compute the number of equilibria of multi-player games. 
\subsection{Preliminaries}
This section presents some preliminary results on the roots of a polynomial that will be used in the subsequent sections. The following lemma is an elementary characterization of stability of equilibrium points of a dynamical system  where the right hand side is a polynomial.
\begin{lemma} 
Consider a dynamical system $\dot{x}=P(x)=a_n x^n+\ldots+ a_1x+ a_0$ where $a_0,\ldots, a_n$ are real coefficients. Suppose that $P$ has $n$ real roots $x_1<x_2<\ldots<x_n$. Then the stability of these equilibrium points is alternatively switched, that is for all $i=1,\ldots n-1$, if $x_i$ is stable then $x_{i+1}$ is unstable and vice versa. In particular, consider the dynamics $\dot{x}=P(x)=Ax^3+Bx^2+Cx+D$. Suppose that $P(x)$ has three real roots $x_1<x_2<x_3$. Then
\begin{enumerate}[(i)]
\item If $A>0$ then $x_2$ is stable; $x_1$ and $x_3$ are unstable.
\item If $A<0$ then $x_2$ is unstable; $x_1$ and $x_3$ are stable.
\end{enumerate}
\end{lemma}
\begin{proof}
We prove the general case since the cubic case is a direct consequence. Since $P$ has $n$ real roots $x_1,\ldots, x_n$, we have $ P(x)=a_n\prod_{i=1}^n(x-x_i)$. Thus
$$
P'(x)=a_n\sum_{i=1}^n\prod_{j\neq i} (x-x_j). 
$$
Therefore for any $i=1,\ldots, n$, we obtain
$$
P'(x_i)=a_n \prod_{j\neq i}(x_i-x_j).
$$
Since $x_1<\ldots<x_n$, we have for any $i=1,\ldots, n-1$,
$$
\mathrm{sign}(P'(x_i))=\mathrm{sign}\Big(a_n (-1)^{n-i}\Big)\quad\text{and}\quad \mathrm{sign}(P'(x_{i+1}))=\mathrm{sign}\Big(a_n (-1)^{n-i-1}\Big)=-\mathrm{sign}(P'(x_i)),
$$
which implies that $P'(x_i)$ and $P'(x_{i+1})$ have alternative signs. Thus their stability is alternative switched. 
\end{proof}
The following lemma specifies the location of roots of a quadratic equation whose proof is omitted. 
\begin{lemma}
\label{lem: quadratic}
Consider a quadratic equation $f(x)=ax^2+bx+c$. Define $\Delta=b^2-4ac$. Then
\begin{enumerate}[(i)]
\item Exactly one of the roots lies in a given interval $(m_1,m_2)$ if $f(m_1)f(m_2)<0$.
\item Both roots are greater than a given number $m$ if 
$$
\Delta\geq 0,\quad -\frac{b}{2a}>m\quad\text{and}\quad a f(m)>0.
$$
\item Both roots are less than a given number $m$ if 
$$
\Delta\geq 0,\quad -\frac{b}{2a}<m\quad\text{and}\quad a f(m)>0.
$$
\item Both roots lie in a given interval $(m_1,m_2)$ if
$$
\Delta\geq 0,\quad m_1<-\frac{b}{2a}<m_2,\quad af(m_1)>0\quad\text{and}\quad af(m_2)>0.
$$
\end{enumerate}
\end{lemma}
\subsection{Two-player games}
\label{sec: 2player}
We first consider the case of two-player games. Let $\{a_{jk}\}_{j,k=1}^n$ be the payoff matrix where $j$ is the strategy of the focal player and $k$ is that of the opponent. Then the average payoff of strategy $j$ and of the whole population are given respectively by
\begin{equation}
\label{eq: average payoff j}
f_j(\x)=\sum_{k=1}^n x_k a_{jk}\quad\text{and}\quad \f(\x)=\sum_{j=1}^n x_j f_j(\x)=\sum_{j,k=1}^n a_{jk}x_jx_k.
\end{equation}
Substituting \eqref{eq: average payoff j} into \eqref{eq: RME} we obtain
\begin{equation}
\label{eq: RME two-player}
\dot{x}_i=\sum_{j,k=1}^n q_{ji} x_jx_k a_{jk}-x_i\sum_{j,k=1}^n a_{jk}x_jx_k.
\end{equation}
In particular, for two-player two-strategy games the replicator-mutator equation is
\begin{multline}
\label{eq: 2-2 games 1}
\dot{x}=q_{11}a_{11}x^2+q_{11}x(1-x)a_{12}+q_{21}x(1-x)a_{21}+q_{21}a_{22}(1-x)^2\\-x\Big(a_{11}x^2+(a_{12}+a_{21})x(1-x)+a_{22}(1-x)^2\Big),
\end{multline}
where $x$ is the frequency of the first strategy and $1-x$ is the frequency of the second one.
Using the identities $q_{11}=q_{22}=1-q, \quad q_{12}=q_{21}=q$, Equation \eqref{eq: 2-2 games 1} becomes
\begin{align}
\dot{x}&=\Big(a_{12}+a_{21}-a_{11}-a_{22}\Big)x^3+\Big(a_{11}-a_{21}-2(a_{12}-a_{22})+q(a_{22}+a_{12}-a_{11}-a_{21})\Big)x^2\nonumber
\\&\quad+\Big(a_{12}-a_{22}+q(a_{21}-a_{12}-2a_{22})\Big)x+q a_{22}.\label{eq: 2-2 games}
\end{align}
The properties of equilibrium points for the case $q=0$ are well-understood, see e.g. \cite{gokhale2014evolutionary}. Thus we consider $0<q\leq 1/2$. In addition, equilibria of  \eqref{eq: 2-2 games} and their stability for the case $a_{11}=a_{22}=1, a_{12}\leq a_{21}\leq 1$ have been studied in \cite{Komarova2010}.
\paragraph{Two-player social dilemma games.}
We first consider two-player social dilemma games. We adopt the following parameterized payoff matrix to study the full space of two-player social dilemma games where the first strategy is cooperator and second is defector \cite{santos:2006pn,wang2015universal}, $a_{11} = 1; \  a_{22} = 0; $
$0 \leq a_{21} = T \leq 2$ and $-1 \leq a_{12} = S \leq 1$, that covers the following games
\begin{enumerate}[(i)]
\item the Prisoner's Dilemma (PD): $2\geq T > 1 > 0 > S\geq -1$,
\item the Snow-Drift (SD) game: $2\geq T > 1 > S > 0$,
\item the Stag Hunt (SH) game: $1 > T  > 0 > S\geq -1$,
\item the Harmony (H) game: $1  > T\geq 0, 1\geq S > 0$.
\end{enumerate}
Note that in the SD-game: $S+T>1$, in the SH-game: $S+T<1$. By simplifying the right hand side of \eqref{eq: 2-2 games}, equilibria of a social dilemma game are roots in the interval $[
0,1]$ of the following cubic equation 
\begin{align}
\Big(T+S-1\Big)x^3+\Big(1-T-2S+q(S-1-T)\Big)x^2+\Big(S+q(T-S)\Big)x = 0.\label{eq: 2-2 games2}
\end{align}
It follows that $x = 0$ is always an equilibrium. If $q=\frac{1}{2}$ then the above equation has two solutions $x_1=\frac{1}{2}$ and $x_2=\frac{T+S}{T+S-1}$. In PD, SD and H-games, $x_2\not \in (0,1)$, thus they have two equilibria $x_0=0$ and $x_1=\frac{1}{2}$. In the SH-game: if $T+S<0$ then the game has three equilibria $x_0=0, x_1=\frac{1}{2}$ and $0<x_2<1$; if $T+S\geq 0$ then the game has only two equilibria $x_0=0, x_1=\frac{1}{2}$.

We consider $q\neq \frac{1}{2}$. For non-zero equilibrium points we solve the following quadratic equation
\begin{equation}
\label{eq: social dilemma}
h(x):=(T+S-1)x^2+(1-T-2S+q(S-1-T))x+S+q(T-S)=:ax^2+bx+c=0.
\end{equation}
Note that  we have $h(1)=-q<0$ for all the above games. In the SD-game, since $T+S-1>0$ and $h(0)=S+q(T-S)=qT+S(1-q)>0$, $h$ is a quadratic and has two positive roots $0<x_1<1<x_2$. Thus the SD-game always has two equilibria: an unstable one $x_0=0$, and a stable one $0<x_1<1$. For the H-game, 
\begin{enumerate}[(i)]
\item If $S+T=1$ then $h$ becomes $h(x)=-(2Tq+1-T)x+(2qT+1-T)-q$ and has a root $x=1-\frac{q}{2Tq+1-T}$. If $q(1-2T)<1-T$ then $x\in (0,1)$ and the game has two equilibria: an unstable one $x_0=0$, and a stable one $0<x_1<1$. If $q(1-2T)\leq 1-T$ then $x<0$ and the game has only one equilibrium $x=0$.
\item if $S+T>1$, then since $h(0)=S+q(T-S)=qT+S(1-q)>0$, $h$ has two roots $0<x_1<1<x_2$; thus the game has two equilibria: an unstable one $x_0=0$, and a stable one $0<x_1<1$.
\item if $S+T<1$ then since $h(0)=S+q(T-S)=qT+S(1-q)>0$, $h$ has two roots $x_2<0<x_1<1$; thus the game has two equilibria: an unstable one $x_0=0$ and a stable one $0<x_1<1$.
\end{enumerate}
Thus the H-game has either $1$ equilibrium or $2$ equilibria. The analysis for the SH-game and the PD-game is more involved since we do not know the sign of $h(0)$.

\textbf{SH-game}. Since $T+S<1$, $h$ is always a quadratic polynomial. Define
\begin{align}
\Delta&=(1-T-2S+q(S-1-T))^2-4(T+S-1)(S+q(T-S)), \label{eq:delta}
\\m&:=-\frac{b}{2a}=\frac{T+2S-1-q(S-T-1)}{2(T+S-1)}=1+\frac{1-T+q(T+1-S)}{2(T+S-1)}.\label{eq: b2a}
\end{align}
Since $T+S-1<0$ and $1-T+q(T+1-S)>0$, we have $m<1$. Applying Lemma \ref{lem: quadratic}, it results in  the following cases: 
\begin{enumerate}[(i)]
\item If $\Delta<0$, then the game has only one equilibrium $x_0=0$ which is stable if $S+q(T-S)<0$ and is unstable if $S+q(T-S)>0$.
\item If $\Delta\geq 0$ and $h(0)>0$, then the game has two equilibria: an unstable one $x_0=0$ and a stable one $0<x_1<1$.
\item If $\Delta\geq 0$ and $h(0)<0$ and $-\frac{b}{2a}>0$ then the game has three equilibria $x_0=0<x_1<x_2<1$ where $x_0$ and $x_2$ are stable while $x_1$ is unstable.
\item If $\Delta\geq 0$ and $h(0)<0$ and $-\frac{b}{2a}<0$ then the game has only one stable equilibrium $x_0=0$.
\end{enumerate}

\textbf{PD-game}. It remains to consider the PD-game. If $S+T=1$ then $h$ becomes $h(x)=-(2Tq+1-T)x+(2qT+1-T)-q$ and has a root $\bar{x}=1-\frac{q}{2Tq+1-T}$. Thus the game has only one equilibrium $x_0=0$ if $\bar{x}\not\in (0,1)$ and has two equilibria if $\bar{x}\in (0,1)$. If $S+T\neq 1$, then $h$ is a quadratic polynomial. Let $\Delta$ and $m$ be defined as in \eqref{eq:delta}-\eqref{eq: b2a}. According to Lemma \ref{lem: quadratic}, we have the following cases: 
\begin{enumerate}[(i)]
\item If $\Delta<0$ then $h$ has no real roots. Thus the game only has one equilibrium $x_0=0$.
\item If $\Delta\geq 0$ and $h(0)=qT+S(1-q)>0$ then $h$ has exactly one root in $(0,1)$. Thus the game has two equilibria.
\item If $\Delta\geq 0,\quad 0<\frac{T+2S-1-q(S-T-1)}{2(T+S-1)}<1, ah(0)=(T+S-1)(qT+S(1-q))>0, \quad\text{and}\quad ah(1)=-q(T+S-1)>0$ then $h$ has two roots in $(0,1)$. Thus the game has three equilibria.
\item In other cases, $h$ has two roots but do not belong to $(0,1)$. Thus the game has only one equilibrium at $x_0=0$. 
\end{enumerate}
For comparison, we consider the case $q=0$. Equation \eqref{eq: 2-2 games2} becomes
$$
(T+S-1)x^3+(1-T-2S)x^2+Sx= x(1-x)(S-(T+S-1)x)=0,
$$
which implies 
$$
x_0=0,~~x_1=1,~~x_2=\frac{S}{T+S-1}.
$$
The condition $0<x_2<1$ is equivalent to 
$$
S(S+T-1)>0\quad\text{and}\quad (1-T)(S+T-1)<0,
$$
which is satisfied in the SD-game and the SH-game but is violated in the PD-game and the H-game. In the SD-game $S+T>1$ and $0=x_1<x_2<1=x_1$, thus $x_2$ is stable and $x_0$ and $x_1$ are unstable. In the SH-game, $S+T<1$ and $0=x_1<x_2<1=x_1$, thus $x_2$ is unstable, $x_1$ and $x_3$ are stable. The PD-game and the H-game have only two equilibria: for the PD-game $x_0=0$ (stable) and $x_1=1$ (unstable), for the H-game: $x_0=0$ (unstable) and $x_1=1$ (stable).

\paragraph{General games.}
Now we consider a general two-player two-strategy game where there is no ranking on the coefficients. An equilibrium point is a root $x\in(0,1)$ of the cubic on the right-hand side of \eqref{eq: 2-2 games 1}
\begin{multline}
\Big(a_{12}+a_{21}-a_{11}-a_{22}\Big)x^3+\Big(a_{11}-a_{21}-2(a_{12}-a_{22})+q(a_{22}+a_{12}-a_{11}-a_{21})\Big)x^2\nonumber
\\+\Big(a_{12}-a_{22}+q(a_{21}-a_{12}-2a_{22})\Big)x+q a_{22}=0
\end{multline}
We define $t:=\frac{x}{1-x}$. Dividing the above equation by $(1-x)^3$ and using the relation $\frac{1}{1-x}=1+t$, the above equation can be written in $t$-variable as
\begin{align*}
P_3(t)&=-a_{11}qt^3 +(a_{12}-a_{21}+q(a_{21}q-a_{11}-a_{12}))t^2+(a_{12}-a_{22}+q(a_{21}+a_{22}-a_{12}))t+a_{22}q
\\&:=a t^3+bt^2+ct+d.
\end{align*}
The number of equilibria of the $2\times 2$-game is equal to the number of positive roots of the cubic $P_3$. Applying Sturm's theorem, see for instance \cite[Theorem 1.4]{Sturmfelds2002}, to the polynomial $P_3$ for the interval $(0,+\infty)$, where  the sign at $+\infty$ of a polynomial is the same as the sign of its leading coefficient, we obtain the following result
\begin{lemma}
Let $s_1$ and $s_2$ be respectively the number of changes of signs in the following sequences
\begin{align*}
&\Big\{d,c,\frac{bc-9ad}{a}, \Delta \Big\},
\\&\Big\{a, \frac{b^2-3ac}{a},\Delta\Big\},
\end{align*}
where $\Delta:=a\big(18 abcd-4b^3 d+b^2c^2-4ac^3-27a^2d^2\big)$ denotes the radicand. Then $P_3$ has exactly $s_1-s_2$ number of positive roots. As consequences
\begin{enumerate}[(i)]
\item  $P_3$ has three distinct real positive roots (thus the game has 3 equilibria) if and only if
$$
\begin{cases}
\Delta> 0,\\ ab<0, \\ac>0,\\ ad<0.
\end{cases} 
$$
\item If there is no change of sign in the sequence of polynomial's coefficients then there is no positive root. That is if 
$$
\begin{cases}
ab>0\\ bc>0\\ cd>0
\end{cases}
$$
then $P_3$ has no positive root (thus the game has no equilibria).
\end{enumerate}
\end{lemma} 
\begin{remark}
\label{rem: q=1/n}
In this remark we show that in the case $q=\frac{n-1}{n}$ the point $\x=(1/n,\ldots, 1/n)$ is always an equilibrium of the general replicator-mutator dynamics regardless of the type of games and of the payoff functions. In fact, since $q=\frac{n-1}{n}$, we have 
$$
q_{ji}=\frac{q}{n-1}=\frac{1}{n}, ~q_{ii}=1-q=\frac{1}{n}.
$$
Substituting this into the formula of $g_i$ in \eqref{eq: RME} we obtain
\begin{align*}
g_i(\x)=\frac{1}{n}\sum_{j=1}^n x_j f_j(\x)-x_i\bar{f}(\x)=(1/n- x_i)\bar{f}(\x).
\end{align*}
Thus the replicator-mutator dynamics always has an \textit{uniform equilibrium} $\x=(1/n,\ldots, 1/n)$, see \cite{Pais2012} for the bifurcation analysis of this equilibrium point for the case $d=2$ and $n\geq 3$.
\end{remark}

\subsection{Muti-player games}
In this section, we focus on the replicator-mutator equation for $d$-player two-strategy games with a symmetric mutation matrix $Q=(q_{ji})$ (with  $j,i\in\{1,2\}$) so that
\[
q_{11}=q_{22}=1-q \quad \text{and}\quad q_{12}=q_{21}=q,
\]
for some constant $0\leq q\leq 1/2$. Note that this is a direct consequence of Equation \eqref{eq: conditionQ} and is not an additional restriction/assumption. Let $x$ be the frequency of $S_1$. Thus the frequency of $S_2$ is $1-x$. The interaction of the individuals in the population is in randomly selected groups of $d$ participants, that is, they play and obtain their fitness from  $d$-player games. Let $a_k$ (resp., $b_k$) be the payoff of an $S_1$-strategist (resp., $S_2$) in a group  containing  other $k$ $S_1$ strategists (i.e. $d-1-k$ $S_2$ strategists). Here we consider symmetric games where the payoffs do not depend on the ordering of the players. In this case, the average payoffs of $S_1$ and $S_2$ are, respectively 
\begin{equation}
\label{eq: fitness}
f_1(x)= \sum\limits_{k=0}^{d-1}a_k\begin{pmatrix}
d-1\\
k
\end{pmatrix}x^k (1-x)^{d-1-k}\quad\text{and}\quad
f_2(x)= \sum\limits_{k=0}^{d-1}b_k\begin{pmatrix}
d-1\\
k
\end{pmatrix}x^k (1-x)^{d-1-k}.
\end{equation} 
The replicator-mutator equation \eqref{eq: RME} then becomes
\begin{align}
\label{eq: RME2}
\dot{x}&=x f_1(x)(1-q)+(1-x) f_2(x)q-x(x f_1(x)+(1-x)f_2(x))\notag
\\&=q\Big[(1-x)f_2(x)-x f_1(x)\Big]+x(1-x)(f_1(x)-f_2(x)).
\end{align}
Note that when $q=0$ we recover the usual replicator equation (i.e. without mutation). In contrast to the replicator equation, $x=0$ and $x=1$ are no longer equilibrium points of the system for $q\neq 0$. In addition, according to Remark \ref{rem: q=1/n} if $q=\frac{1}{2}$ then $x=\frac{1}{2}$ is always an equilibrium point.

Equilibrium points are those points $0\leq x\leq 1$ that make the right-hand side of \eqref{eq: RME2} vanish, that is
\begin{equation}
\label{eq: equilibria1}
q\Big[(1-x)f_2(x)-x f_1(x)\Big]+x(1-x)(f_1(x)-f_2(x))=0.
\end{equation}
Using \eqref{eq: fitness}, Eq. \eqref{eq: equilibria1} becomes
\begin{multline}
\label{eq: equilibria2}
q\bigg[\sum_{k=0}^{d-1}b_k\begin{pmatrix}
d-1\\k
\end{pmatrix} x^k(1-x)^{d-k}-\sum_{k=0}^{d-1}a_k\begin{pmatrix}
d-1\\k
\end{pmatrix} x^{k+1}(1-x)^{d-1-k}\bigg]
\\+\sum_{k=0}^{d-1}\beta_k \begin{pmatrix}
d-1\\k
\end{pmatrix} x^{k+1}(1-x)^{d-k}=0,
\end{multline}
where $\beta_k:=a_k-b_k$. Now setting $t:=\frac{x}{1-x}$ then dividing \eqref{eq: equilibria2} by $(1-x)^{d+1}$ and using the relation that $(1+t)=\frac{1}{1-x}$, we obtain
\begin{equation}
q(1+t)\Big[\sum_{k=0}^{d-1}b_k\begin{pmatrix}
d-1\\k
\end{pmatrix}t^k-\sum_{k=0}^{d-1}a_k\begin{pmatrix}
d-1\\k
\end{pmatrix}t^{k+1}\Big]+ \sum_{k=0}^{d-1}\beta_k \begin{pmatrix}
d-1\\k
\end{pmatrix} t^{k+1}=0.
\end{equation}
By regrouping terms and changing the sign, we obtain the following polynomial equation in $t$-variable
\begin{equation}
\label{eq: P}
P(t):=\sum_{k=0}^{d+1}c_k t^k=0,
\end{equation}
where the coefficient $c_k$ for $k=0,\cdots, d+1$ is given by
\begin{equation}
\label{eq: c}
c_k:=\begin{cases}
-qb_0~~\text{for}~~k=0,\\
(q-1)(a_0-b_0)-q(d-1)b_1~~\text{for}~~k=1,\\
q a_{k-2}\begin{pmatrix}
d-1\\k-2
\end{pmatrix}+(q-1)(a_{k-1}-b_{k-1})\begin{pmatrix}
d-1\\k-1
\end{pmatrix}-qb_k\begin{pmatrix}
d-1\\k
\end{pmatrix}~~\text{for}~~ k=2,\ldots, d-1,\\
(q-1)(a_{d-1}-b_{d-1})+qa_{d-2}(d-1)~~\text{for}~~k=d,\\
qa_{d-1}~~\text{for}~~ k=d+1.
\end{cases}
\end{equation}
Thus the number of equilibria of $d$-player two-strategy games is the same as the number of positive roots of the polynomial $P$. We now use Decartes' rule of signs to count the latter. Let $\c:=\{c_0,c_1,\ldots, c_{d+1}\}$ be the sequence of coefficients given in \eqref{eq: c}. Applying Decartes' rule of signs we obtain the following result.

\begin{lemma}
The number of positive roots of $P$, which is also the number of equilibria of the $d$-player two-strategy replicator-mutator dynamics, is either equal to the number of sign changes of $\c$ or is less than it by an even amount.
\end{lemma}

In \cite{PENA2014}, the author has employed a similar approach to study the number of equilibria for the standard replicator dynamics, in which $P$ turns out to be a Bernstein polynomial and many useful properties of Bernstein polynomials were exploited. In the following remark we show that the polynomial $P$ can also be written in the form of a Bernstein polynomial.  
\begin{remark}
\label{re: Bernstein}
Using the identities B.4 and B.5 in \cite{Pena2015} we can write $q[(1-x)f_2(x)-xf_1(x)]$ as a polynomial in Bernstein form of degree $d$ (call it $P_1(x)$), and similarly $x(1-x)(f_1(x)-f_2(x))$ as a polynomial in Bernstein form of degree $d+1$ (call it $P_2(x)$), as follows 
\begin{equation*}
P_1(x) = q[(1-x)f_2(x)-xf_1(x)]  = q \left( \sum\limits_{k=0}^{d}\frac{b_k (d-k)  - k a_{k-1}}{d} \begin{pmatrix}
d\\
k
\end{pmatrix}x^k (1-x)^{d-k} \right),
\end{equation*} 
\begin{equation*}
P_2(x) = x(1-x)(f_1(x)-f_2(x)) =  \sum\limits_{k=0}^{d+1}\frac{k(d+1-k)(  a_{k-1} -  b_{k-1} )}{d(d+1)}  \begin{pmatrix}
d+1\\
k
\end{pmatrix}x^k (1-x)^{d+1-k}.
\end{equation*} 
Using the following identity, obtained by multiplying the polynomial by $((1-x) + x)$,

\begin{equation*}
\sum\limits_{k=0}^{d}c_k\begin{pmatrix}
d\\
k
\end{pmatrix}x^k (1-x)^{d+1-k}  = \sum\limits_{k=0}^{d+1} \frac{ (d+1-k) c_k + k c_{k-1} }{d+1}\begin{pmatrix}
d+1\\
k
\end{pmatrix}x^k (1-x)^{d+1-k},
\end{equation*} 
we have 
\begin{multline*}
P_1(x)  = q \Big(  \sum\limits_{k=0}^{d+1} \frac{ (d+1-k) [(d-k)b_k   - k a_{k-1}] + k [b_{k-1} (d+1-k)  - (k-1) a_{k-2}] }{d(d+1)}\\
\times \begin{pmatrix}
d+1\\
k
\end{pmatrix}x^k (1-x)^{d+1-k} \Big).
\end{multline*} 
Combining the above computations, we have converted $P_1(x) + P_2(x)$ into a polynomial in Bernstein form  
\begin{equation*}
P_1(x) + P_2(x)= \frac{1}{d(d+1)} \sum\limits_{k=0}^{d+1} \rho_k\begin{pmatrix}
d+1\\
k
\end{pmatrix}x^k (1-x)^{d+1-k}, 
\end{equation*} 
where 
\begin{align*}
\rho_k&= k(d+1-k)(  a_{k-1} -  b_{k-1} ) + q \Big((d+1-k) [(d-k)b_k   - k a_{k-1}] 
\\ &\hspace*{6cm}+ k [b_{k-1} (d+1-k)  - (k-1) a_{k-2}] \Big)\\
&= q (d+1-k)(d-k)b_k +  (1-q) (d+1-k)k  (a_{k-1} - b_{k-1}) - q k (k-1) a_{k-2}.
\end{align*}
Direct computations show that (note that we have changed the sign of $c_k$ for notation convenience in the subsequent sections)
$$
\rho_k=-\frac{c_k d(d+1)}{\begin{pmatrix}
d+1\\k
\end{pmatrix}.
}$$
Having written $P$ in the form of a Bernstein polynomial, similar general results on the equilibrium points of the replicator-mutator dynamics as in \cite{PENA2014} could be, in principle, obtained using the link between the sign pattern of the sequence $\boldsymbol{\rho}=\{\rho_0,\ldots, \rho_{d+1}\}$ and the sign pattern and number of
roots of the polynomial $P$. We do not go into further details here and leave this interesting topic for future research.
\end{remark}

For a (real) polynomial $P$ we denote by $S(P)$ the number of changes of signs in the sequence of coefficients of $P$ disregarding zeros and by $R(P)$ the number of positive roots of $P$ counted with multiplicities. Decartes' rule of signs only provides an upper bound for $R(P)$ in terms of $S(P)$.  Recently it has been shown that $R(P)$ can be computed exactly as $S(PQ)$ for some polynomial $Q$ or as a limit of $S((t+1)^n P(t))$ as $n$ tends to infinity.
\begin{theorem}\cite{Avendano2010}
\label{thm: Decartes1}
Let $P$ be a non-zero real polynomial. 
\begin{enumerate}[(i)]
\item There exists a real polynomial $Q$ with all non-negative coefficients such that $S(PQ)=R(P)$.
\item The sequence $S((t+1)^n P(t))$ is monotone decreasing with limit equal to $R(P)$.
\end{enumerate}
\end{theorem}
The polynomial $Q$ in part (i) involves all the roots of $P$ (even the imaginary ones), which are not known in general, hence part (i) is practically inefficient. The sequence $\{S((t+1)^n P(t))\}_{n}$ can be easily computed, but it only can be used for approximating $R(P)$.
Note that for $P(t)=c_{d+1} t^{d+1}+\ldots+c_1 t+c_0$, we have
$$
(t+1)^n P(t)=\sum_{j=0}^n\sum_{i=0}^{d+1} c_i\begin{pmatrix}
n\\j
\end{pmatrix}t^{i+j}=\sum_{k=0}^{n+d+1}\sum_{i=0}^{d+1} c_i\begin{pmatrix}
n\\k-i
\end{pmatrix}t^k.
$$
Thus thus $k$-th coefficient of $(t+1)^n P(t)$ is
\begin{equation}
\label{eq: c2}
a^k_n=\sum_{i=0}^{d+1} c_i \begin{pmatrix}
n\\k-i
\end{pmatrix}.
\end{equation}
\begin{corollary}
\label{cor: sn}
 Let $s_n$ be the number of changes of signs in the sequence $\{a^k_n\}_{k=0}^{n+d+1}$ defined in \eqref{eq: c2}. Then the  number $N$ of equlibria of a $d$-player two-strategy game is
\begin{equation}
N=R(P)=\lim_{n\rightarrow \infty} s_n.
\end{equation}
\end{corollary}
\begin{figure}
\centering
\includegraphics[width = 0.9\linewidth]{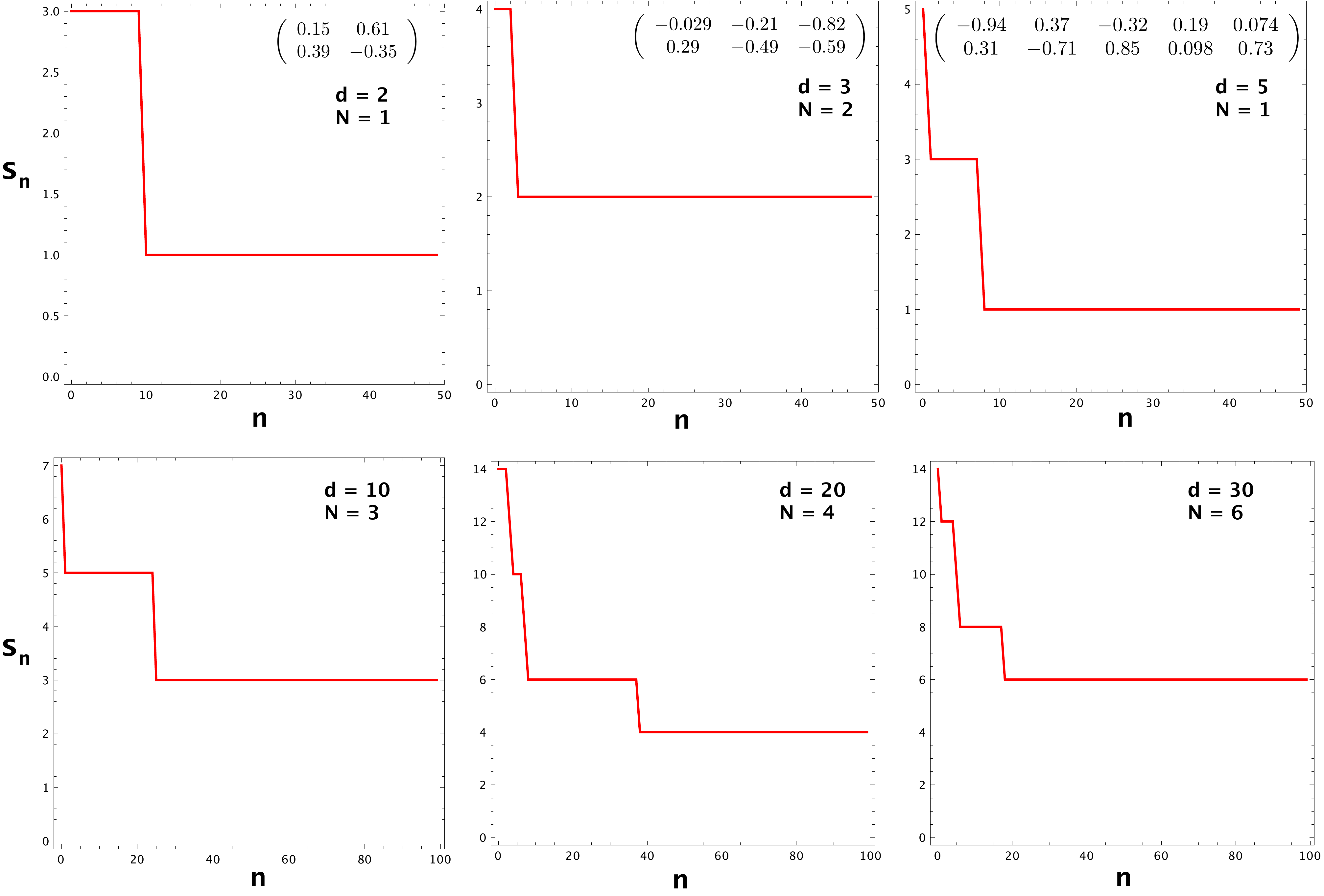}
\caption{\textbf{Plot of $s_n$ for some randomly chosen payoff matrices} (we adopted $q = 0.1$ in all cases). We indicate the number players $d$ in the game,  the payoff matrix used for small $d$ (for the sake of representation given  large sizes of the payoff matrices for large $d$), 
and the number of internal equilibria, $N$.  For sufficiently large $n$, $s_n$ decreasingly converges to the corresponding value of $N$. }
\label{fig:sn}
\end{figure} 
Corollary \ref{cor: sn} provides us with a simple method to calculate the number of equilibria, $N$, for a given $d$-player two-strategy game. In Figure \ref{fig:sn}, we show a number of examples. The value of $n$ such that $s_n$ reaches $N$ varies significantly for different games and is usually (very) large. It would be an interesting problem to find the smallest value of $n$ satisfying $s_n=N$. An upper-bound for such $n$ is also helpful. This is still  an open problem \cite{Avendano2010}. However, in the particular case when $P$ has no positive root, we have the following theorem.
\begin{theorem}\cite{Powers2001} Let $P(t)=c_{d+1} x^{d+1}+\ldots+c_1 t+c_0$.
\label{thm: Decartes3}
If $R(P)=0$ then $S((t+1)^{n_0} P(t))=0$ where
\begin{equation*}
n_0=\left\lceil \begin{pmatrix}
d+1\\2
\end{pmatrix}\frac{\max_{0\leq i\leq d+1}\big\{c_i/\begin{pmatrix}
d+1\\i
\end{pmatrix}\big\}}{\min_{\lambda\in[0,1]}\big\{(1-\lambda)^{d+1} f(\frac{\lambda}{1-\lambda})\big\}}-d-1\right\rceil.
\end{equation*}
\end{theorem}
\begin{corollary}
If $S((t+1)^{n_0}P(t))\geq 1$ then $R(P)\geq 1$.
\end{corollary}
\section{Properties of equilibrium points: random games}
\label{sec: random}
In this section we study random games. For two-player social dilemma games, we calculate the probability of having a certain number of equilibria when $S$ and $T$ are uniformly distributed. For multi-player games, we compute the expected number of equilibria when the payoff entries are normally distributed.
\subsection{Probability of having a certain number of equilibria in social dilemma games}
We consider two-player social dilemma games in Section \ref{sec: 2player} but $T$ and $S$ are now random variables uniformly distributed in the corresponding intervals. In this section, 
$p^G_k$, where $G \in \{SD, H, SH, PD\}$ and $k\in\{1,2,3\}$, denotes the probability of a game $G$ having $k$ equilibria. According to the analysis of Section \ref{sec: 2player}, all of the games have at least one equilibrium at the origin. In addition, the SD-game always has two equilibria, that is
$$
p_1^{SD}=p_3^{SD}=0, \quad p_2^{SD}=1.
$$
We also know that the H-game has either $1$ or $2$ equilibria. The probability that it has $1$ equilibrium is smaller than the probability that $S+T=1$. Since $S+T$ has a continuous density function, it implies that $p_1^{H}=0$. Thus we also have 
$$
p_1^{H}=p_3^{H}=0,\quad p_2^{H}=1.
$$
For the SH-game and PD-game, we are able to calculate the probability of having two equilibria explicitly since its condition on $T$ and $S$ is simple which depends only on a convex combination of $T$ and $S$.  The conditions on $S$ and $T$ for these games to have $1$ equilibrium or $3$ equilibria are much more complex since they involve $\Delta$ defined in \eqref{eq:delta}, which is a nonlinear function of $S$ and $T$.

\textbf{SH-game}. Suppose that $S\sim U([-1,0]),~~T\sim U([0,1])$. Then
\begin{align*}
&qT\sim U([0,q]), \quad f_{qT}(x)=\begin{cases}\frac{1}{q} \quad\text{if}\quad 0\leq x\leq q,\\
0 \quad\text{otherwise}
\end{cases};
\\& (1-q) S\sim U([q-1,0]),\quad f_{(1-q)S}(y)=\begin{cases}\frac{1}{(1-q)}\quad\text{if}\quad q-1\leq y\leq 0,\\
0 \quad\text{otherwise}.
\end{cases}
\end{align*}
We now compute $p_2^{SH}$ explicitly. The probability that the SH-game has two equilibria, $p_2^{SH}$, is the probability that $h(0)h(1)<0$. Since $h(1)<0$ we have
\begin{equation}
\label{eq: p2SH}
p_2^{SH}=\mathrm{Prob}(h(0)>0)=\mathrm{Prob}(qT+S(1-q)>0)=\int_0^\infty f_Z^{SH}(x)\,dx,
\end{equation}
where $f_Z^{SH}$ is the probability density function of the random variable $Z:=qT+(1-q)S$, which is given by
\begin{align*}
f_Z^{SH}(x)&=(f_{qT}\ast f_{(1-q)S})(x)=\int_{-\infty}^\infty f_{qT}(x-y) f_{(1-q)S}(y)\,dy
\\&=\frac{1}{1-q}\int_{q-1}^0 f_{qT}(x-y)\,dy
\\&\overset{(*)}{=}\frac{1}{1-q}\begin{cases}
\int_{q-1}^{x}\frac{1}{q}\,dy\quad\text{if}~~q-1\leq x\leq 2q-1,\\
\int_{x-q}^{x}\frac{1}{q}\,dy\quad\text{if}~~2q-1\leq x\leq 0,\\
\int_{x-	q}^0\frac{1}{q}\,dy\quad\text{if}~~0\leq x\leq q,\\
0\quad\text{otherwise}
\end{cases}
\\&=\frac{1}{1-q}\begin{cases}
\frac{x+1-q}{q}\quad\text{if}~~q-1\leq x\leq 2q-1,\\
1\quad\text{if}~~2q-1\leq x\leq 0,\\
\frac{q-x}{q}\quad\text{if}~~0\leq x\leq q,\\
0\quad\text{otherwise}.
\end{cases}
\end{align*}
Note that to obtain $(*)$, we use the fact that $f_{qT}(x-y)$ is $1/q$ if $0\leq x-y\leq q$ and is zero otherwise. Thus the domain of the integral is restricted to
$$
D=\{(x,y):~~ q-1\leq y\leq 0~\&~~0\leq x-y\leq q\},
$$
which gives rise to the cases in $(*)$. Substituting the formula of $f_Z$ into \eqref{eq: p2SH} we obtain
$$
p_2^{SH}=\int_0^\infty f^{SH}_Z(x)\,dx=\frac{1}{1-q}\int_0^q \frac{q-x}{q}\,dx=\frac{q}{2(1-q)}.
$$
It follows that $q\mapsto p_2^{SH}$ is an increasing function. We plot this function in Figure \ref{fig:p2}.

\textbf{PD-game}. Suppose that $T\sim U([1,2])$ and $S\sim U([-1,0])$. Then
\begin{align*}
&qT\sim U([q,2q]),\quad f_{qT}=\begin{cases}
\frac{1}{q}\quad\text{if}\quad q\leq x\leq 2q,\\
0\quad \text{otherwise};
\end{cases}
\\& (1-q) S\sim U([q-1,0]),\quad f_{(1-q)S}(y)=\begin{cases}\frac{1}{(1-q)}\quad\text{if}\quad q-1\leq y\leq 0,\\
0 \quad\text{otherwise}.
\end{cases}
\end{align*}
Similarly as in \eqref{eq: p2SH} we have
\begin{equation*}
p_2^{PD}=\int_0^\infty f_Z^{PD}(x)\,dx,
\end{equation*}
where $f_Z^{PD}$ is the probability density function of  $Z=qT+(1-q)S$. To calculate this function, we need to consider two different cases $0< q\leq 1/3$ (hence $q-1\leq -2q\leq -q<0$) and $1/3\leq q\leq 1/2$ (hence $-2q\leq q-1\leq -q<0$).   For $0< q\leq 1/3$ we have
\begin{align*}
f_Z^{PD}(x)&=(f_{qT}\ast f_{(1-q)S})(x)=\int_{-\infty}^\infty f_{qT}(x-y) f_{(1-q)S}(y)\,dy
\\&=\frac{1}{1-q}\int_{q-1}^0 f_{qT}(x-y)\,dy
\\&=\frac{1}{1-q}\begin{cases}
\int_{q-1}^{x-q}\frac{1}{q}\,dy\quad\text{if}\quad 2q-1\leq x\leq 3q-1,\\
\int_{x-2q}^{x-q}\frac{1}{q}\,dy\quad\text{if}\quad 3q-1\leq x\leq q,\\
\int_{x-2q}^0\frac{1}{q}\,dy\quad\text{if}\quad q\leq x\leq 2q,\\
0\quad\text{otherwise}
\end{cases}
\\&=\frac{1}{1-q}\begin{cases}
\frac{x+1-2q}{q}\quad\text{if}\quad 2q-1\leq x\leq 3q-1,\\
1\quad\text{if}\quad 3q-1\leq x\leq q,\\
\frac{2q-x}{q}\quad\text{if}\quad q\leq x\leq 2q,\\
0\quad\text{otherwise}
\end{cases}
\end{align*}
Hence for $0\leq q\leq 1/3$, we have
\begin{align*}
p_2^{PD}&=\int_0^\infty f_Z^{PD}(x)\,dx=\int_0^{q}f_Z^{PD}(x)\,dx+\int_{q}^{2q}f_Z^{PD}(x)\,dx=\frac{1}{1-q}\Big(\int_0^q 1\,dx+
\int_q^{2q}\frac{2q-x}{q}\,dx\Big)
\\&=\frac{3q}{2(1-q)}.
\end{align*}
\begin{figure}
\centering
\includegraphics[width = 0.6\linewidth]{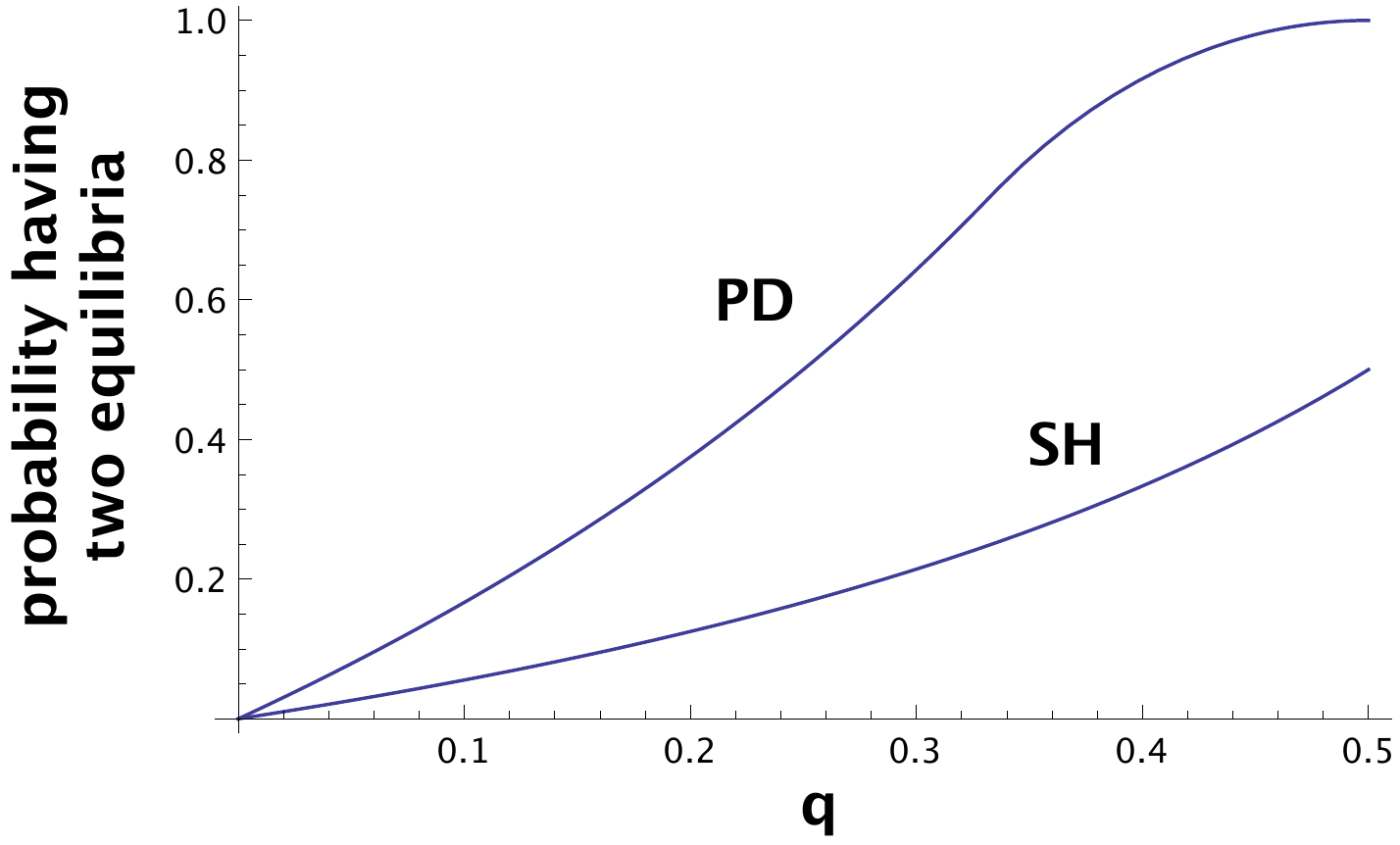}
\caption{\textbf{Probability of having two equilibrium points for Prisoner's Dilemma (PD) and Stag Hunt (SH) games}, according to analytical results obtained in Section \ref{sec: random}. Both functions are increasing; $p_2^{PD}$ is always bigger than $p_2^{SD}$; the maximum of $p_2^{PD}$ is $1$ while the maximum of $p_2^{SD}$ is $1/2$. These results also corroborate  the  simulation results using samplings in Figure \ref{fig:social-dilemmas}.  }
\label{fig:p2}
\end{figure} 
\begin{figure}
\centering
\includegraphics[width = \linewidth]{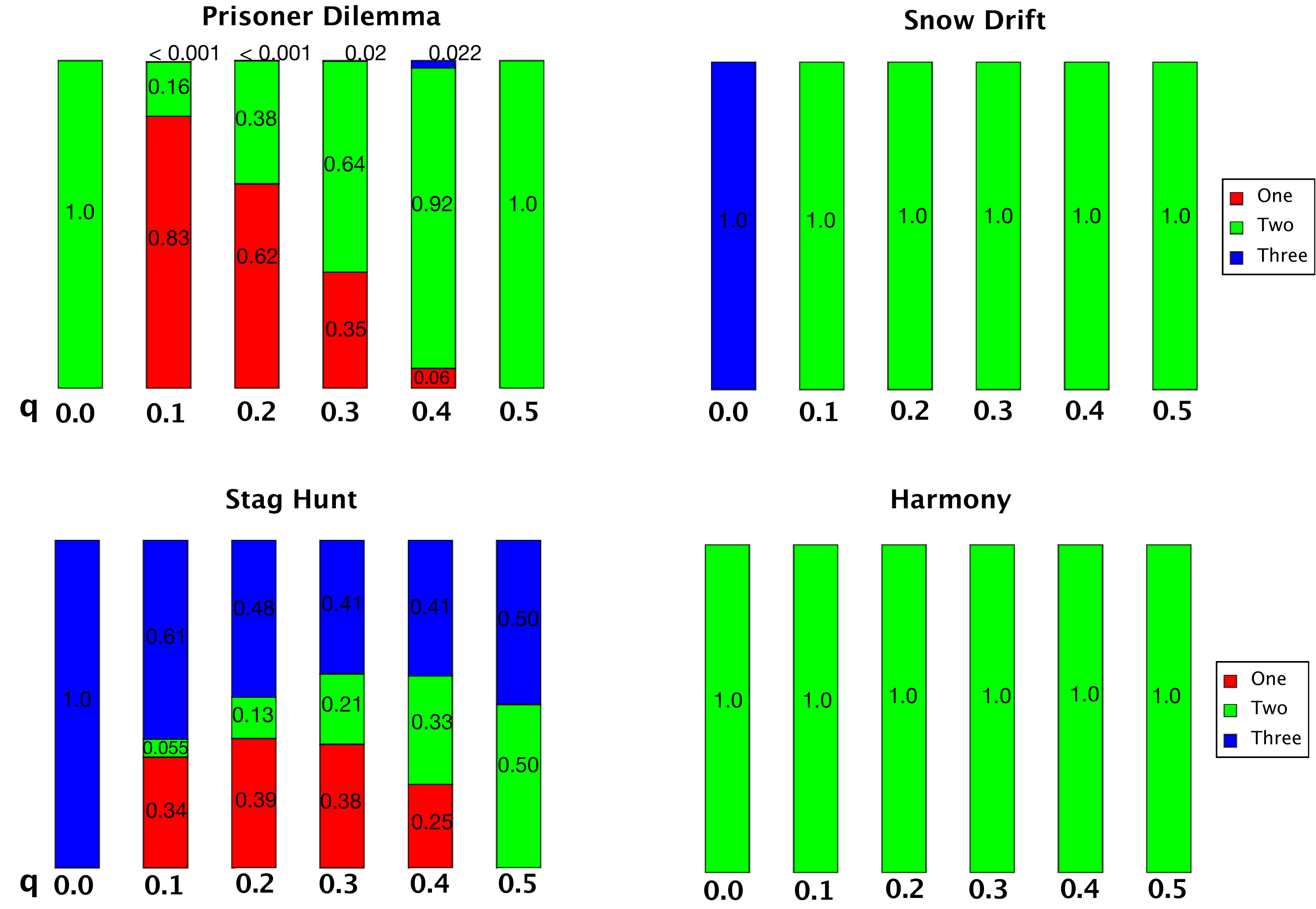}
\caption{\textbf{Probabilities of observing  a certain number of equilibrium points for each social dilemma game, for different mutation strengths, $q$}.  $S$ and $T$ are drawn from uniform distributions. The results are averaged over sampling   $10^6$ pairs of $S$ and $T$  drawn from the corresponding ranges in a social dilemma.
 All results are obtained using Mathematica. }
\label{fig:social-dilemmas}
\end{figure} 
For $1/3\leq q\leq 1/2$, we have
\begin{align*}
f_Z^{PD}(x)&=\frac{1}{1-q}\int_{q-1}^0f_{qT}(x-y)\,dy
\\&=\frac{1}{1-q}\begin{cases}
\int_{q-1}^{x-q}\frac{1}{q}\,dy\quad\text{if}\quad 2q-1\leq x\leq 3q-1,\\
\int_{x-2q}^{x-q}\frac{1}{q}\,dy\quad\text{if}\quad 3q-1\leq x\leq q,\\
\int_{x-2q}^0\frac{1}{q}\,dy\quad\text{if}\quad q\leq x\leq 2q,\\
0\quad\text{otherwise}
\end{cases}
\\&=\frac{1}{1-q}\begin{cases}
\frac{x+1-2q}{q}\quad\text{if}\quad 2q-1\leq x\leq 3q-1,\\
1\quad\text{if}\quad 3q-1\leq x\leq q,\\
\frac{2q-x}{q}\quad\text{if}\quad q\leq x\leq 2q,\\
0\quad\text{otherwise}
\end{cases}
\end{align*}
Hence for $1/3\leq q\leq 1/2$, we have
\begin{align*}
p_2^{PD}&=\int_0^\infty f_Z^{PD}(x)\,dx=\int_0^{3q-1}f_Z^{PD}(x)\,dx+\int_{3q-1}^{q}f_Z^{PD}(x)\,dx+\int_q^{2q}f_Z^{PD}(x)\,dx
\\&=\frac{1}{1-q}\Big(\int_0^{3q-1}\frac{x+1-2q}{q} \,dy+
\int_{3q-1}^q 1\,dy+\int_q^{2q}\frac{2q-x}{q}\Big)
\\&=3-\frac{1}{2q(1-q)}.
\end{align*}
In summary, we obtain
$$
p_2^{PD}=\begin{cases}
\frac{3 q}{2(1-q)}\quad \text{if}\quad 0<q\leq 1/3,\\
3-\frac{1}{2q(1-q)}\quad \text{if}\quad 1/3\leq q\leq 1/2. 
\end{cases}
$$
It follows that $q\mapsto p_2^{PD}$ is also increasing. We also plot this function in Figure \ref{fig:p2}. Moreover, in Figure \ref{fig:social-dilemmas}, we numerically compute the probability of having a certain number of equilibria for each game by averaging over $10^6$ samples of $T$ and $S$. The numerical results are in accordance with the analytical computations. In the H-game: $p_2=1$ (hence $p_1=p_3=0$) for all values of $q$. In the SD-game: when $q=0$, $p_3=1$ (hence $p_1=p_2=0$) but  $p_2=1$ (hence $p_1=p_3=0$) for all $q>0$. In the PD-game: when $q=0$, $p_2=1$ (hence $p_1=p_3=0$) but when $0<q<1/2$ all $p_1,p_2,p_3$ are positive although $p_3$ is very small; $p_2$ is increasing and attains its maximum $1$ at $q=1/2$. In the SH-game: when $q=0$, $p_3=1$ (hence $p_1=p_2=0$). When $0<q<1/2$, the picture is more diverse: all $p_1,p_2$ and $p_3$ are non-negligible; $p_2$ is increasing and attains its maximum $1/2$ at $q=1/2$. Moreover, note that for $q > 0$, there is at least one equilibrium ($x = 0$) in all cases, where the remaining ones are internal equilibria. 
To the contrary, when $q = 0$, PD and H games always  have two non-internal equilibria (at $x = 0$ and $x = 1$) while SH and SG games have three equilibria (two non-internal  and one internal). With mutation ($q > 0$), $x = 1$ is no longer an equilibrium in all cases. 
Therefore, the SD-game has the same number of internal equilibria (one) while  it  gains one more internal equilibrium in H game. 
In the PD-game, the probability of having at least one internal equilibrium increases with $q$. 
In the SH-game, the probability of having two internal (i.e. gaining one more compared to the no mutation case) is  high. 
 In short, except for the SD game, introducing mutation leads to the probability of gaining an additional internal equilibrium (thus increasing behavioural diversity) in all social dilemmas. This probability is 100\% in the H-game, increases with $q$ in the PD-game (reaching 100\% when $q = 0.5$) and  is roughly 40-60\% in the SH-game.

\subsection{Expected number of equilibria of multi-player two strategy games}
We recall that finding an equilibrium point of the replicator-mutator dynamics for $d$-player two-strategy games is equivalent to finding a positive root of the polynomial \eqref{eq: P} with coefficients given in \eqref{eq: c}. In this section, by employing techniques from random polynomial theory, we provide explicit formulas for the computation of the expected number of internal equilibrium points of the replicator-mutator dynamics  where the entries of the payoff matrix are random variables, thus extending our previous results for the replicator dynamics \cite{DH15,DuongHanJMB2016,DuongTranHanJMB,DuongTranHanDGA}. We will apply the following result on the expected number of positive roots of a general random polynomial.
\begin{theorem}\cite[Theorem 3.1]{EK95}
\label{thm: EK95}
Consider a random polynomial
$$
Q(x)=\sum_{i=0}^n \alpha_k x^k,
$$
where $\{\alpha_k\}_{0\leq k\leq n}$ are the elements of a multivariate normal distribution with mean zero and covariance matrix $C$. Then the expected number of positive roots of $Q$ is given by
\begin{equation}
\label{eq: expected roots}
E_Q=\frac{1}{\pi}\int_{0}^\infty\Big(\frac{\partial^2}{\partial x\partial y}\big(\log v(x)^T C v(y)\big)\big\vert_{y=x=t}\Big)^\frac{1}{2}\,dt,
\end{equation}
where
\begin{equation*}
v(x)=\begin{pmatrix}
1\\x\\ \vdots\\ x^n
\end{pmatrix},\quad v(y)=\begin{pmatrix}
1\\y\\ \vdots \\ y^n
\end{pmatrix}.
\end{equation*}
\end{theorem}
Defining
$$
H(x,y)=\sum_{i,j=0}^n C_{ij} x^i y^j, \quad M(t)=H(t,t),\quad A(t)=\partial^2_{xy} H(x,y)\vert_{y=x=t}, \quad B(t)=\partial_x H(x,y)\vert_{y=x=t},
$$
then $E_Q$ can be written as
\begin{equation}
\label{eq: E}
E_Q=\frac{1}{\pi}\int_0^\infty \frac{\sqrt{A(t)M(t)-B(t)^2}}{M(t)}\,dt.
\end{equation}
We now apply Theorem \ref{thm: EK95} to the random polynomial $P$ given in \eqref{eq: P} and obtain formulas for the expected number of equilibria of the replicator-mutator dynamics for $d$-player two-strategy games. It turns out that the case $q=0.5$ needs special treatment since according to Remark \ref{rem: q=1/n} $x=1/2$ is always an equilibrium point.
\subsubsection{The case $q\neq 0.5$}
Suppose that $a_k$ and $b_k$ are independent standard normally distributed random variables with mean zero. Then, for $q\neq \frac{1}{2}$, the random vector $\c=\{c_0,\ldots, c_{d+1}\}$ defined in \eqref{eq: c} has a (symmetric) covariance matrix $C=(C_{ij})_{0\leq i,j\leq d+1}$ given by
\begin{align*}
& C_{kk}=\begin{cases}
q^2~~\text{for}~~k=0,\\
2 (q-1)^2+q^2(d-1)^2~~\text{for}~~k=1,\\
q^2 \begin{pmatrix}
d-1\\k-2
\end{pmatrix}^2+2(q-1)^2\begin{pmatrix}
d-1\\k-1
\end{pmatrix}^2+q^2\begin{pmatrix}
d-1\\k
\end{pmatrix}^2~~\text{for}~~k=2,\ldots, d-1,\\
2(q-1)^2+ q^2(d-1)^2~~\text{for}~~k=d,\\
q^2 ~~\text{for}~~k={d+1};
\end{cases}
\\& C_{k k+1}=\begin{cases}
q(q-1)~~\text{for}~~k=0,\\
q(q-1)+q(q-1)(d-1)^2~~\text{for}~~k=1,\\
q(q-1)\begin{pmatrix}
d-1\\k-1
\end{pmatrix}^2+q(q-1)\begin{pmatrix}
d-1\\k
\end{pmatrix}^2~~\text{for}~~k=2,\ldots, d-2,\\
q(q-1)(d-1)^2+q(q-1)~~\text{for}~~k=d-1,\\
q(q-1)~~\text{for}~~k=d;
\end{cases}
\\& C_{i j}=0~~\text{for}~~0\leq i<j\leq d+1~: j-i\geq 2.
\end{align*}
Using the convention that whenever $k<0$ or $k>n$ then $
\begin{pmatrix}
n\\k
\end{pmatrix}=0
$, we can simplify $C$ as
\begin{align}
& C_{kk}=q^2 \begin{pmatrix}
d-1\\k-2
\end{pmatrix}^2+2(q-1)^2\begin{pmatrix}
d-1\\k-1
\end{pmatrix}^2+q^2\begin{pmatrix}
d-1\\k
\end{pmatrix}^2~~\text{for}~~ k=0,\ldots, d+1, \label{eq: Ckk}
\\& C_{kk+1}=q(q-1)\begin{pmatrix}
d-1\\k-1
\end{pmatrix}^2+q(q-1)\begin{pmatrix}
d-1\\k
\end{pmatrix}^2, ~~\text{for}~~ k=0,\ldots, d, \label{eq: Ckk+1}
\\& C_{i j}=0~~\text{for}~~0\leq i<j\leq d+1~: j-i\geq 2. \label{eq: Cij}
\end{align}
Applying Theorem \ref{thm: EK95} we obtain the following result.
\begin{proposition}
Suppose that $a_k$ and $b_k$ are independent standard normally distributed random variables with mean zero and that $q\neq 0.5$. We define
\begin{align*}
H(x,y)&=\sum_{k=0}^{d+1}C_{kk}x^k y^k+\sum_{k=0}^d C_{kk+1}(x^k y^{k+1}+x^{k+1}y^k),\\
M(t)&=H(t,t),\quad A(t)=\partial^2_{xy}H(x,y)\big\vert_{y=x=t},\quad B(t)=\partial_{x}H(x,y)\big\vert_{y=x=t},
\end{align*}
where the coefficient $C_{ij}, 0\leq i,j \leq d+1$ are given in \eqref{eq: Ckk}, \eqref{eq: Ckk+1} and \eqref{eq: Cij}. Then the expected number of equilibria of a $d$-player two-strategy replicator-mutator dynamics is given by
$$
E=\frac{1}{\pi}\int_0^\infty \frac{\sqrt{A(t)M(t)-B^2(t)}}{M(t)}\,dt.
$$
\end{proposition}
\begin{figure}
\centering
\includegraphics[width = 0.8\linewidth]{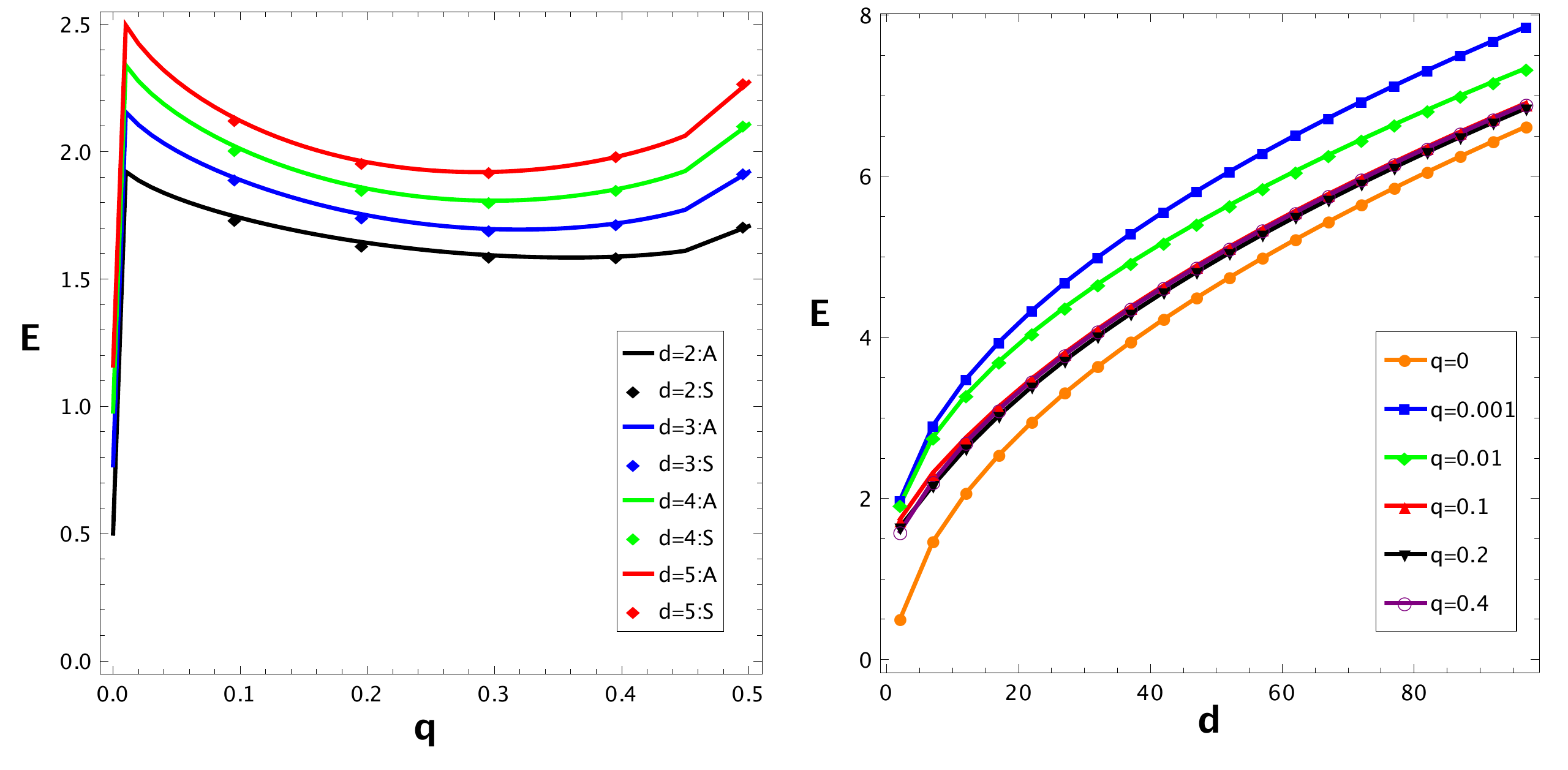}
\caption{(\textit{Left panel}) \textbf{Analytical vs. simulation sampling results of the average number of internal equilibrium points ($E$)} for varying $q$ and for different values of $d$.
 The solid lines are generated from analytical (\textbf{A}) formulas of $E$.   The solid diamonds capture simulation (\textbf{S}) results obtained by averaging over $10^6$ samples of the payoff entries (normal distribution). Analytical and simulations results are in accordance with each other.
(\textit{Right panel}) \textbf{Plot of $E$} for increasing $d$ and for different values of $q$. In general,  $E$ increases with $d$. $E$ is always larger when $q > 0$  than when $q = 0$. Also, $E$ is largest  when $q$ is close to 0 (i.e. rare mutation).  
 All results are obtained using Mathematica. }
\label{fig:pm theory vs samplings}
\end{figure}

\subsubsection{The case $q=0.5$}
The case $q=0.5$ needs to be treated differently since in this case, according to Remark \ref{rem: q=1/n},  $x=1/2$ is always an equilibrium. Other equilibrium points are roots of the average fitness of the whole population $\bar{f}(x)=0$ due to Remark \ref{rem: q=1/n}, that is
\begin{equation*}
0=\bar{f}(x)=x f_1(x)+(1-x)f_2(x)\overset{\eqref{eq: fitness}}{=}\sum_{k=0}^{d-1}a_k\begin{pmatrix}
d-1\\k
\end{pmatrix} x^{k+1} (1-x)^{d-1-k}+\sum_{k=0}^{d-1}b_k\begin{pmatrix}
d-1\\k
\end{pmatrix} x^{k} (1-x)^{d-k}.
\end{equation*}
Since $x=1$ is not a solution, by dividing the right-hand side of the above equation by $(1-x)^d$, and let $t:=\frac{x}{1-x}$ then we obtain the following equation
\begin{align*}
P(t)&=\sum_{k=0}^{d-1}a_k\begin{pmatrix}
d-1\\k
\end{pmatrix} t^{k+1}+\sum_{k=0}^{d-1}b_k\begin{pmatrix}
d-1\\k
\end{pmatrix} t^{k}
\\&=\sum_{k=0}^d\left[a_{k-1}\begin{pmatrix}
d-1\\k-1
\end{pmatrix}+b_k\begin{pmatrix}
d-1\\k
\end{pmatrix}\right] t^k
\\&=:\sum_{k=0}^d c_k t^k,
\end{align*}
where $$ 
c_k=a_{k-1}\begin{pmatrix}
d-1\\k-1
\end{pmatrix}+b_k\begin{pmatrix}
d-1\\k
\end{pmatrix},\quad\text{for}~k=0,\ldots, d.
$$
Suppose that $a_k$ and $b_k$ are independent standard normally distributed random variables with mean zero. Then the random vector $\c=\{c_0,\ldots, c_{d+1}\}$ has a (symmetric) covariance matrix $C=(C_{ij})_{0\leq i,j\leq d+1}$ given by
\begin{equation*}
C_{ij}=\left(\begin{pmatrix}
d-1\\k-1
\end{pmatrix}^2+\begin{pmatrix}
d-1\\k
\end{pmatrix}^2\right)\, \delta_{ij},
\end{equation*}
where $\delta_{ij}$ is the Kronecker delta. Applying Theorem \ref{thm: EK95} and noticing that $x=1/2$ is always an equilibrium,  we obtain the following result.
\begin{proposition}
Suppose that $a_k$ and $b_k$ are independent standard normally distributed random variables with mean zero and that $q=0.5$. We define
\begin{align*}
H(x,y)&=\sum_{k=0}^{d}\left(\begin{pmatrix}
d-1\\k-1
\end{pmatrix}^2+\begin{pmatrix}
d-1\\k
\end{pmatrix}^2\right)x^k y^k,\\
M(t)&=H(t,t),\quad A(t)=\partial^2_{xy}H(x,y)\big\vert_{y=x=t},\quad B(t)=\partial_{x}H(x,y)\big\vert_{y=x=t},
\end{align*}
Then the expected number of equilibria of a $d$-player two-strategy replicator-mutator dynamics is given by
$$
E=1+\frac{1}{\pi}\int_0^\infty \frac{\sqrt{A(t)M(t)-B^2(t)}}{M(t)}\,dt.
$$
\end{proposition}
In Figure \ref{fig:pm theory vs samplings} we show that the results obtain from analytical formulas of $E$ corroborate with those obtained from numerical simulations by averaging over a large number of randomly generated payoff matrices. Figure \ref{fig:pm theory vs samplings} also reveals that the expected number of equilibria exhibits several interesting behaviours. We will elaborate more on this point in Section \ref{sec: summary}. 
\begin{figure}
\centering
\includegraphics[width = 0.5\linewidth]{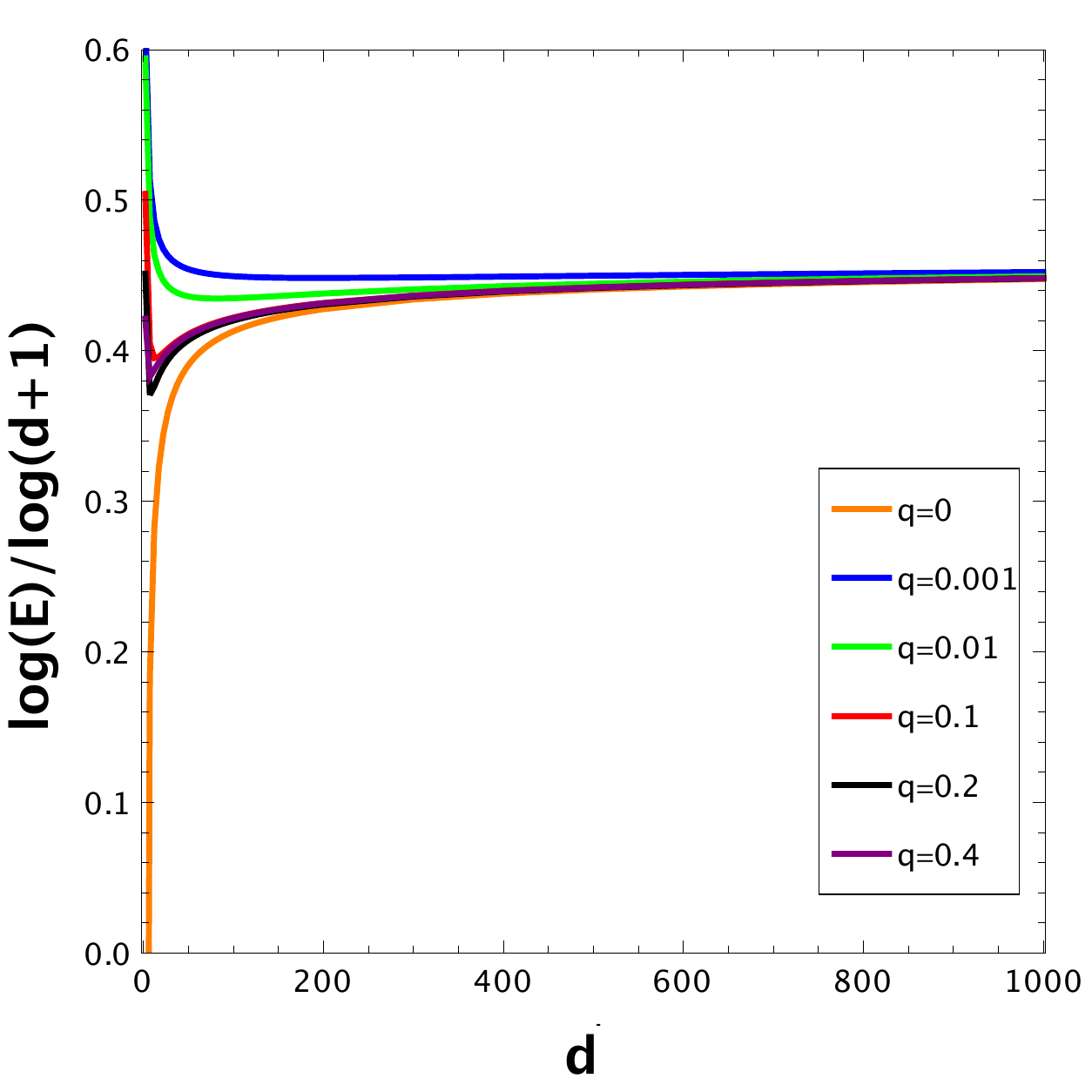}
\caption{\textbf{Plot of $log(E)/log(d+1)$ for varying $d$}. For different values of $q$, this quantity converges to the same value. 
 All results are obtained using Mathematica. }
\label{fig:logE}
\end{figure} 
\section{Conclusion and outlook}
\label{sec: summary}
Understanding equilibrium properties of the replicator-mutator dynamics for multi-player multi-strategy games is a difficult problem due to its complexity: to find an equilibrium, one needs to solve a system of multivariate polynomials. In this paper,  employing techniques from classical and random polynomial theory, we study the number of equilibria for both deterministic and random two-strategy games. For deterministic games,  using Decartes' rule of signs and its recent developments, we provide a method to compute the number of equilibria via the sign changes of the coefficients of a polynomial. For two-player social dilemma games, we compute the probability of observing a certain number of equilibria when the payoff entries are uniformly distributed. For multi-player two-strategy random games whose payoffs are independently distributed according to a normal distribution, we obtain explicit formulas to compute the expected number of equilibria by relating it to the expected number of positive roots of a random polynomial. We also perform numerical simulations to compare with and to illustrate our analytical results. We observe that $E$   is always larger in the presence of mutation (i.e. when $q > 0$) than when mutation is absent (i.e. when $q = 0$), implying  that mutation leads to larger behavioural diversity in a dynamical system (see again Figure \ref{fig:pm theory vs samplings}). Interestingly, $E$ is largest  when $q$ is close to 0 (i.e. rare mutation), rather than when it is  large. 
In general, our findings might have important implications for the understanding of social and biological diversities, where biological mutations and behavioural errors are present, i.e. in the study of evolution of cooperative behaviour and population fitness distribution \cite{Levin2000,pena2012group, santos2012role}. Furthermore, numerical simulations also suggest a number of open problems that we leave for future work.

\textit{Asymptotic behaviour of the expected number of equilibria when the number of players tends to infinity.} In \cite{DuongHanJMB2016}, we proved that 
\begin{equation}
\label{eq: replicator limit}
\lim\limits_{d\rightarrow\infty}\frac{\ln E(d)}{\ln(d-1)}=\frac{1}{2},
\end{equation}
where $E(d)$ is the expected number of internal equilibria of the replicator dynamics for $d$-player two-strategy games, in which the payoff entries are randomly distributed. To obtain \eqref{eq: replicator limit}, we utilized several useful connections to Legendre's polynomials. In Figure \ref{fig:logE}, we plot $\frac{\ln E(q,d)}{\ln (d+1)}$, where $E(q,d)$ is the expected number of equilibria for the replicator-mutator dynamics, as a function of $d$ for various values of $q$. We observe that they all converge to the same limit as $d$ tends to infinity, but in different manner: for $q=0$, it increasingly approaches  the limit while for $q > 0$ sufficiently small, at first they are decreasing and then for sufficiently large $d$, they also increasingly approach to the limit. Thus, it is expected that there is a phase transition. Proving this rigorously would be an interesting problem. The method used in \cite{DuongHanJMB2016} seems not to be working since there is no direct connections to Legendre's polynomials.

\textit{Asymptotic behaviour of the expected number of equilibria when the mutation tends to zero.} The classical replicator dynamics is obtained from the replicator-mutator dynamics by setting the mutation to be zero. Thus it is a natural question to ask how a certain quantity (such as the expected number of equilibria) behaves when the mutation tends to zero. Both Figures \ref{fig:pm theory vs samplings} and  \ref{fig:logE} demonstrate that the expected number of equilibria changes significantly when the mutation is turned on. In addition, using explicit formulas of the probability of observing two equilibria for the SH-game and the PD-game obtained in Section \ref{sec: random}, we clearly see a jump when $q$ approaches zero:
$$
\lim_{q\rightarrow 0}p_2^{q, SH,PD}=0\neq 1=p_2^{0,SH,PD}.
$$
Both observations suggest that these quantities exhibit singular behaviour at $q=0$. Characterizing this behaviour would be a challenging problem for future work.

\textit{Bifurcation phenomena of the replicator-mutator dynamics for multi-player games}. \\
In~\cite{Pais2012}, the authors proved Hopf bifurcations for the replicator-mutator dynamics with $d=2$ and $n\geq 3$ and characterized the existence of stable limit cycles using an analytical derivation of the Hopf bifurcations points and the corresponding first Lyapunov coefficients. In addition, they also showed that the limiting behaviors
are tied to the structure of the fitness model. Another interesting topic for further research would be to extend the results of \cite{Pais2012} to multi-player games.
\section*{Acknowledgment}  We would like to thank anonymous referees for useful suggestions which help us  improve the presentation of the paper, in particular Remark \ref{re: Bernstein} was suggested to us by one of the referee. T.A.H. also acknowledges support from Future of Life Institute (grant RFP2-154).
\bibliographystyle{alpha}

\begin{thebibliography}{HDMHS02}

\bibitem[AW09]{azais2009level}
Azais, J.M. and Wschebor, M.
\newblock Level Sets and Extrema of Random Processes and Fields
\newblock Wiley, 2009.


\bibitem[Av10]{Avendano2010}
M.~Avenda\~no.
\newblock Descartes' rule of signs is exact!
\newblock {\em J. Algebra}, 324(10):2884--2892, 2010.

\bibitem[BCV97]{broom:1997aa}
M.~Broom, C.~Cannings, and G.T. Vickers.
\newblock Multi-player matrix games.
\newblock {\em Bull. Math. Biol.}, 59(5):931--952, 1997.

\bibitem[Bro00]{Broom2000}
M.~Broom.
\newblock Bounds on the number of esss of a matrix game.
\newblock {\em Mathematical Biosciences}, 167(2):163 -- 175, 2000.

\bibitem[DH15]{DH15}
M.~H. Duong and T.~A. Han.
\newblock On the expected number of equilibria in a multi-player multi-strategy
  evolutionary game.
\newblock {\em Dynamic Games and Applications}, pages 1--23, 2015.

\bibitem[DH16]{DuongHanJMB2016}
M.~H. Duong and T.~A. Han.
\newblock Analysis of the expected density of internal equilibria in random
  evolutionary multi-player multi-strategy games.
\newblock {\em Journal of Mathematical Biology}, 73(6):1727--1760, 2016.

\bibitem[DTH18]{DuongTranHanDGA}
M.~H. Duong, H.~M. Tran, and T.~A. Han.
\newblock On the expected number of internal equilibria in random evolutionary
  games with correlated payoff matrix.
\newblock {\em Dynamic Games and Applications}, Jul 2018.

\bibitem[DTH19]{DuongTranHanJMB}
M.~H. Duong, H.~M. Tran, and T.~A. Han.
\newblock On the distribution of the number of internal equilibria in random
  evolutionary games.
\newblock {\em Journal of Mathematical Biology}, 78(1):331--371, Jan 2019.

\bibitem[EK95]{EK95}
A.~Edelman and E.~Kostlan.
\newblock How many zeros of a random polynomial are real?
\newblock {\em Bull. Amer. Math. Soc. (N.S.)}, 32(1):1--37, 1995.

\bibitem[FH92]{fudenberg:1992bv}
D.~Fudenberg and C.~Harris.
\newblock Evolutionary dynamics with aggregate shocks.
\newblock {\em J. Econ. Theory}, 57:420--441, 1992.

\bibitem[GF13]{Galla2013}
T.~Galla and J.~D. Farmer.
\newblock Complex dynamics in learning complicated games.
\newblock {\em Proceedings of the National Academy of Sciences},
  110(4):1232--1236, 2013.

\bibitem[GRLD09]{gross2009generalized}
T.~Gross, L.~Rudolf, S.~A Levin, and U.~Dieckmann.
\newblock Generalized models reveal stabilizing factors in food webs.
\newblock {\em Science}, 325(5941):747--750, 2009.

\bibitem[GT10]{gokhale:2010pn}
C.~S. Gokhale and A.~Traulsen.
\newblock Evolutionary games in the multiverse.
\newblock {\em Proc. Natl. Acad. Sci. U.S.A.}, 107(12):5500--5504, 2010.

\bibitem[GT14]{gokhale2014evolutionary}
C.~S. Gokhale and A.~Traulsen.
\newblock Evolutionary multiplayer games.
\newblock {\em Dynamic Games and Applications}, 4(4):468--488, 2014.

\bibitem[Had81]{Hadeler1981}
K.~P. Hadeler.
\newblock Stable polymorphisms in a selection model with mutation.
\newblock {\em SIAM Journal on Applied Mathematics}, 41(1):1--7, 1981.

\bibitem[HDMHS02]{hauert:2002te}
C.~Hauert, S.~De~Monte, J.~Hofbauer, and K.~Sigmund.
\newblock Volunteering as red queen mechanism for cooperation in public goods
  games.
\newblock {\em Science}, 296:1129--1132, 2002.

\bibitem[HPL17]{HanJaamas2016}
T.A. Han, L.~M. Pereira, and T.~Lenaerts.
\newblock Evolution of commitment and level of participation in public goods
  games.
\newblock {\em Autonomous Agents and Multi-Agent Systems}, 31(3):561--583,
  2017.

\bibitem[HTG12]{HTG12}
T.~A. Han, A.~Traulsen, and C.~S. Gokhale.
\newblock On equilibrium properties of evolutionary multi-player games with
  random payoff matrices.
\newblock {\em Theoretical Population Biology}, 81(4):264 -- 272, 2012.

\bibitem[IFN05]{Imhof-etal2005}
L.~A. Imhof, D.~Fudenberg, and M.~A. Nowak.
\newblock Evolutionary cycles of cooperation and defection.
\newblock {\em Proceedings of the National Academy of Sciences},
  102(31):10797--10800, 2005.

\bibitem[KL10]{Komarova2010}
N.~L. Komarova and S.~A. Levin.
\newblock Eavesdropping and language dynamics.
\newblock {\em Journal of Theoretical Biology}, 264(1):104 -- 118, 2010.

\bibitem[KNN01]{Komarova2001JTB}
N.~L. Komarova, P.~Niyogi, and M.~A. Nowak.
\newblock The evolutionary dynamics of grammar acquisition.
\newblock {\em Journal of Theoretical Biology}, 209(1):43 -- 59, 2001.

\bibitem[Kom04]{Komarova2004}
N.~L. Komarova.
\newblock Replicator-mutator equation, universality property and population
  dynamics of learning.
\newblock {\em Journal of Theoretical Biology}, 230(2):227 -- 239, 2004.

\bibitem[Lev00a]{Levin2000}
S.~A. Levin.
\newblock Multiple scales and the maintenance of biodiversity.
\newblock {\em Ecosystems}, 3(6):498--506, Nov 2000.

\bibitem[NKN01]{Nowaketal2001}
M.~A. Nowak, N.~L. Komarova, and P.~Niyogi.
\newblock Evolution of universal grammar.
\newblock {\em Science}, 291(5501):114--118, 2001.

\bibitem[{Olf}07]{Olfati2007}
R.~{Olfati-Saber}.
\newblock Evolutionary dynamics of behavior in social networks.
\newblock In {\em 2007 46th IEEE Conference on Decision and Control}, pages
  4051--4056, Dec 2007.

\bibitem[PCNeL12]{Pais2012}
D.~Pais, C.~Caicedo-N\'{u}n\~{e}z, and N.~Leonard.
\newblock Hopf bifurcations and limit cycles in evolutionary network dynamics.
\newblock {\em SIAM Journal on Applied Dynamical Systems}, 11(4):1754--1784,
  2012.

\bibitem[Pe{\~n}12]{pena2012group}
J.~Pe{\~n}a.
\newblock Group-size diversity in public goods games.
\newblock {\em Evolution}, 66(3):623--636, 2012.

\bibitem[PJR{\etalchar{+}}17]{perc2017statistical}
M.~Perc, J.~J. Jordan, D.~G. Rand, Z.~Wang, S.~Boccaletti, and A.~Szolnoki.
\newblock Statistical physics of human cooperation.
\newblock {\em Physics Reports}, 687:1--51, 2017.

\bibitem[PLN14]{PENA2014}
J.~Pe{\~n}a, L.~Lehmann, and G.~N{\"o}ldeke.
\newblock Gains from switching and evolutionary stability in multi-player
  matrix games.
\newblock {\em Journal of Theoretical Biology}, 346:23 -- 33, 2014.

\bibitem[PLN15]{Pena2015}
J.~Pe{\~n}a, G.~N{\"o}ldeke and L.~Lehmann.
\newblock Evolutionary dynamics of collective action in spatially structured populations.
\newblock {\em Journal of Theoretical Biology}, 382, 122-136, 2015.

\bibitem[PR01]{Powers2001}
V.~Powers and B.~Reznick.
\newblock A new bound for p\'{o}lya's theorem with applications to polynomials
  positive on polyhedra.
\newblock {\em Journal of Pure and Applied Algebra}, 164(1):221 -- 229, 2001.
\newblock Effective Methods in Algebraic Geometry.

\bibitem[SPLP12]{santos2012role}
F.~C Santos, F.~L. Pinheiro, T.~Lenaerts, and J.~M. Pacheco.
\newblock The role of diversity in the evolution of cooperation.
\newblock {\em Journal of theoretical biology}, 299:88--96, 2012.

\bibitem[SPL06]{santos:2006pn}
F.~C. Santos, J.~M. Pacheco, and T.~Lenaerts.
\newblock Evolutionary dynamics of social dilemmas in structured heterogeneous
  populations.
\newblock {\em Proc. Natl. Acad. Sci. U.S.A.}, 103:3490--3494, 2006.

\bibitem[SS92]{StadlerSchuster1992}
P.~F. Stadler and P.~Schuster.
\newblock Mutation in autocatalytic reaction networks.
\newblock {\em Journal of Mathematical Biology}, 30(6):597--632, Jun 1992.

\bibitem[Stu02]{Sturmfelds2002}
\newblock Solving Systems of Polynomial Equations.
\newblock {\em American Mathematical Society, CBMS regional conferences series}, No. 97, 2002. 

\bibitem[WKJT15]{wang2015universal}
Z.~Wang, S.~Kokubo, M.~Jusup, and J.~Tanimoto.
\newblock Universal scaling for the dilemma strength in evolutionary games.
\newblock {\em Physics of life reviews}, 14:1--30, 2015.

\end{thebibliography}
\newcommand{\etalchar}[1]{$^{#1}$}

\end{document}